\documentclass[draft,leqno]{amsart}

\usepackage[Symbol]{upgreek}

\usepackage{amssymb}
\usepackage{epic}
\usepackage{mathrsfs}
\usepackage{accents}
\usepackage{stmaryrd}

\input{xy}
\xyoption{matrix}
\xyoption{arrow}

\numberwithin{equation}{section}

\newcommand{\bbold}{\mathbb}

\def\R { {\bbold R} }
\def\Q { {\bbold Q} }

\def\N { {\bbold N} }
\def\T { {\bbold T} }

\def \ex{\operatorname{e}}

\renewcommand\epsilon{\varepsilon}

\def \<{\langle}
\def \>{\rangle}

\def \((  {(\!(}
\def \)) {)\!)}

\DeclareMathSymbol{\precequ}{\mathrel}{symbols}{"16}
\DeclareMathSymbol{\succequ}{\mathrel}{symbols}{"17}

\def \nasymp{\not\asymp}

\newtheorem{theorem}{Theorem}[section]
\newtheorem{lemma}[theorem]{Lemma}

\newtheorem{cor}[theorem]{Corollary}

\theoremstyle{definition}

\theoremstyle{remark}

\let\oldi\i
\let\oldj\j
\renewcommand\i{\relax\ifmmode{\boldsymbol{i}}\else\oldi\fi}
\renewcommand\j{\relax\ifmmode{\boldsymbol{j}}\else\oldj\fi}

\renewcommand\leq{\leqslant}
\renewcommand\geq{\geqslant}
\renewcommand\preceq{\preccurlyeq}
\renewcommand\succeq{\succcurlyeq}
\renewcommand\le{\leq}
\renewcommand\ge{\geq}

\DeclareMathAlphabet{\mathbf}{OML}{cmm}{b}{it}

\DeclareFontFamily{U}{fsy}{}
\DeclareFontShape{U}{fsy}{m}{n}{<->s*[.9]psyr}{}
\DeclareSymbolFont{der@m}{U}{fsy}{m}{n}
\DeclareMathSymbol{\der}{\mathord}{der@m}{182}

\DeclareSymbolFont{der@m}{U}{fsy}{m}{n}
\DeclareMathSymbol{\derdelta}{\mathord}{der@m}{100}

\DeclareSymbolFont{imag@m}{OT1}{cmr}{m}{ui}
\DeclareMathSymbol{\imag}{\mathord}{imag@m}{105}

\DeclareFontFamily{OMS}{smallo}{}
\DeclareFontShape{OMS}{smallo}{m}{n}{<->s*[.65]cmsy10}{}
\DeclareSymbolFont{smallo@m}{OMS}{smallo}{m}{n}
\DeclareMathSymbol{\smallo}{\mathord}{smallo@m}{79}

\DeclareFontFamily{OMS}{largerdot}{}
\DeclareFontShape{OMS}{largerdot}{m}{n}{<->s*[.8]cmsy10}{}
\DeclareSymbolFont{largerdot@m}{OMS}{largerdot}{m}{n}
\DeclareMathSymbol{\largerdot}{\mathord}{largerdot@m}{15}

\DeclareMathSymbol{\llambda}{\mathord}{der@m}{108}
\DeclareMathSymbol{\rrho}{\mathord}{der@m}{114}

\def \Upg{\Upgamma}
\def \upl{\uplambda}
\def \Upl{\Uplambda}
\def \upo{\upomega}

\newcommand{\equationqed}[1]{\[\pushQED{\qed}#1 \qedhere\popQED\]\let\qed\relax}
\newcommand{\alignqed}[1]{\begin{align*}\pushQED{\qed} #1 \qedhere\popQED\end{align*}\let\qed\relax}

\makeatletter
\newcommand{\dminus}{\mathbin{\text{\@dminus}}}

\newcommand{\@dminus}{%
  \ooalign{\hidewidth\raise1ex\hbox{\bf.}\hidewidth\cr$\m@th-$\cr}%
}
\makeatother

\def \C{\mathcal{C}}

\def \Calinf{\mathcal{C}^{<\infty}}
\def \Calr{\mathcal{C}^{r}}

\begin{document}

\title{On a Differential Intermediate Value Property}

\author[Aschenbrenner]{Matthias Aschenbrenner}
\address{Kurt G\"odel Research Center for Mathematical Logic\\
Universit\"at Wien\\
1090 Wien\\ Austria}
\email{matthias.aschenbrenner@univie.ac.at}

\author[van den Dries]{Lou van den Dries}
\address{Department of Mathematics\\
University of Illinois at Urbana-Cham\-paign\\
Urbana, IL 61801\\
U.S.A.}
\email{vddries@math.uiuc.edu}

\author[van der Hoeven]{Joris van der Hoeven}
\address{CNRS, LIX, \'Ecole Polytechnique\\
91128 Palaiseau Cedex\\
France}
\email{vdhoeven@lix.polytechnique.fr}

\begin{abstract} Liouville closed $H$-fields are ordered differential fields whose ordering and derivation interact in a natural way and where every linear differential equation of order $1$ has a nontrivial solution. (The introduction gives a precise definition.) For a Liouville closed $H$-field $K$ with small derivation we show:
$K$ has the Intermediate Value Property for differential polynomials iff $K$ is elementarily equivalent to the ordered differential field of transseries. We also indicate how this applies to Hardy fields.   
\end{abstract}

\thanks{The first-named author was partially supported by NSF Grant DMS-1700439. We thank Allen Gehret for commenting on an earlier version of this paper.}

\date{May 2021}

\maketitle

\section*{Introduction}\label{intro}

\noindent
Throughout this introduction $K$ is an ordered differential field, that is, an ordered field equipped with a derivation $\der\colon K \to K$. (We usually write $f'$ instead of $\der f$, for~$f\in K$.)  
 Its constant field $$C\ :=\ \{f\in K:\, f'=0\}$$ yields the (convex) valuation ring 
 $$\mathcal{O}\ :=\ \big\{f\in K:\,  \text{$|f|\le c$ for some $c\in C$}\big\}$$ of $K$, with maximal ideal 
 $$\smallo\ :=\ {\big\{f\in K:\,  \text{$|f|< c$ for all $c>0$ in $C$}\big\}}.$$
 (It may help to think of the elements of $K$ as germs of real valued functions and of~$f\in \mathcal{O}g$ and $f\in \smallo g$ as $f=O(g)$ and $f=o(g)$, respectively.) 
 The above definitions exhibit $C$, $\mathcal{O}$, and $\smallo$ as definable in $K$ in the sense of model theory. 

Key example: the ordered differential field $\T$ of {\bf transseries}, which contains $\R$ as an ordered subfield, and where $C=\R$. We refer to
~\cite{ADH} for the rather elaborate construction of $\T$ and for any fact about $\T$ that gets mentioned without proof. 

Other important examples are Hardy fields. (Hardy~\cite{Ha} proved a striking theorem on logarithmic-exponential functions. Bourbaki~\cite{Bou} put this into the general setting of what they called Hardy fields.) Here we can give a definition from scratch that doesn't take much space.
Notation: $\C$ is the ring of germs at $+\infty$ of continuous real-valued functions on halflines $(a,+\infty)$, $a\in \R$. For $r=1,2,\dots$, let
$\Calr$ be the subring of $\C$ consisting of the germs at $+\infty$ of $r$-times continuously differentiable real-valued functions on such halflines. This yields the subring
$$\Calinf\ :=\  \bigcap_{r\in \N^{\ge 1}} \Calr$$
of $\C$, and $\Calinf$ is naturally a {\em differential\/} ring. For a germ $f\in \C$ we let $f$ also denote any real valued function representing this germ,
if this causes no ambiguity. A real number is identified with the germ of the corresponding constant function: $\R\subseteq \C$.

\medskip\noindent
A {\bf Hardy field} is by definition a differential subfield of $\Calinf$. 
{\em Examples}: 
$$\quad \Q,\\ \quad \R,\quad \R(x),\quad \R(x, \ex^x),\quad \R(x, \ex^x, \log x),\quad \R(\Gamma, \Gamma', \Gamma'',\dots),$$
where $x$ denotes the germ at $+\infty$ of the identity function on $\R$. All these are actually {\em analytic\/} Hardy fields, that is, its elements are germs of real analytic functions.

Let $H$ be a Hardy field. Then $H$ is an {\em ordered\/} differential field:
for $f\in H$, either $f(x) >0$ eventually (in which case we set $f>0$), or $f(x)=0$, eventually,
or $f(x) < 0$, eventually; this is because $f\ne 0$ in $H$ implies $f$ has a multiplicative inverse in $H$, so $f$ cannot have arbitrarily large zeros. Also, if $f'<0$, then $f$ is eventually strictly decreasing; if $f'=0$, then $f$ is eventually constant; 
if $f'>0$, then $f$ is eventually strictly increasing.

\medskip\noindent
In order to state the main result of this paper we need a bit more terminology:  an {\bf $H$-field\/}  is a $K$ (that is, an ordered differential field) such that:  \begin{itemize}
\item for all $f\in K$, if $f> C$, then $f'>0$;
\item $\mathcal{O}= C + \smallo$ (so $C$ maps isomorphically onto the residue field $\mathcal{O}/\smallo$). 
\end{itemize}
We also say that  {\bf $K$ has small derivation\/} if for all $f\in \smallo$ we have $f'\in \smallo$. Hardy fields have small derivation, and any Hardy field containing $\R$ is an $H$-field.

An $H$-field $K$ is said to be
{\bf Liouville closed\/} if it is real closed and for every~$f\in K$ there are $g,h\in K^\times$ such that $f=g'=h'/h$. 
The ordered differential field $\T$ is a Liouville closed $H$-field with small derivation.
Any Hardy field $H\supseteq \R$ has a smallest (with respect to inclusion) Liouville closed
Hardy field extension $\text{Li}(H)$.  
(The notions of ``$H$-field'' and ``Liouville closed $H$-field'' are introduced in \cite{AvdD1}. The capital $H$ is in honor of Hardy, Hausdorff, and Hahn, who pioneered various aspects of our topic about a century ago, as did Du Bois-Reymond and Borel even earlier.)

Now a very strong property: we say $K$ {\bf has DIVP} (the Differential
Intermediate Value Property) if for every polynomial $P\in K[Y_0,\dots, Y_r]$ and all $f<g$ in $K$
with $$P(f, f',\dots, f^{(r)})\ <\ 0\  <\  P(g, g',\dots, g^{(r)})$$ there exists $y\in K$  such that 
$f<y<g$ and $P(y, y',\dots, y^{(r)})=0$.
(Existentially closed ordered differential fields have DIVP by~\cite{Singer} and ~\cite[Proposition~1.5]{Spodzieja}; this has limited interest for us since the ordering and derivation in those structures do not interact.) 
Actually, DIVP is a bit of an afterthought: in \cite{ADH} we considered instead two robust but rather technical
properties, $\upo$-freeness and newtonianity, and proved that $\T$ is $\upo$-free and newtonian. (One can think of newtonianity as a variant of differential-henselianity.) Afterwards we saw that ``$\upo$-free + newtonian'' is equivalent to DIVP, for Liouville closed
$H$-fields. Our aim is to establish this equivalence:  Theorem~\ref{DIVP}, the main result of this short paper.

\medskip\noindent
We did not consider DIVP in \cite{ADH}, but it is surely an appealing property and easier to grasp than the more fundamental notions of
$\upo$-freeness and newtonianity. (The latter make sense in a wider setting of valued differential fields where the valuation does not necessarily arise from an ordering, as is
the case for $H$-fields.) 

Besides \cite{ADH} we shall rely on~\cite{JvdH}, which focuses on a particular ordered differential subfield of $\T$, namely $\T_{\operatorname{g}}$, consisting of the so-called {\em grid-based\/} transseries; see also~\cite[Appendix~A]{ADH}. We summarize what we need from~\cite{JvdH} as follows:

\medskip
\noindent
{\em $\T_{\operatorname{g}}$ is a newtonian $\upo$-free Liouville closed $H$-field with small derivation, and $\T_{\operatorname{g}}$ has~\rm{DIVP}}. We alert the reader that the terms {\em newtonian\/} and
{\em $\upo$-free\/} do not occur in ~\cite{JvdH}, and that $\T_{\operatorname{g}}$ there is denoted by $\T$.

\medskip\noindent
We call attention to the fact that $K$ is a Liouville closed $H$-field iff $K\models \operatorname{LiH}$ for a set $\text{LiH}$ (independent of $K$) of sentences in the language
of ordered differential fields. Also, for $H$-fields, ``$\upo$-free'' is expressible by a single sentence in the language of ordered differential fields, and ``newtonian'' as well as ``DIVP" by a set of sentences in this language.
The reason that ``$\upo$-free + newtonian'' is central in ~\cite{ADH} are various theorems proved there, which are also relevant here. To state these theorems, we consider 
an $H$-field $K$ below as an 
$\mathcal{L}$-structure, where 
$$\mathcal{L}\ :=\ \{\,0,\,1,\, {+},\, {-},\, {\times},\, \der,\, {<},\, {\preceq}\,\}$$ is the language of ordered valued differential fields. The symbols~$0$,~$1$,~${+}$,~${-}$,~${\times}$,~$\der$,~$<$ name the usual primitives of $K$, and $\preceq$ encodes its valuation: 
for $a,b\in K$, 
$$ a\preceq b\quad :\Longleftrightarrow\quad a\in \mathcal{O} b.$$
We can now summarize what we need from~\cite[Chapters~15, 16]{ADH} as follows:

\medskip
\noindent
{\it The theory of  newtonian $\upo$-free Liouville closed $H$-fields is model complete, and is the model companion of the theory of $H$-fields. The theory of  newtonian $\upo$-free Liouville closed $H$-fields whose derivation is small is complete and has $\T$ as a model}.

\medskip\noindent
For an $H$-field $K$ its valuation ring $\mathcal{O}$ and so the binary relation $\preceq$ on $K$ can be defined in terms of the other primitives by an {\em existential\/} formula independent of $K$. However, by~\cite[Corollary 16.2.6]{ADH} this cannot be done by a universal such formula and so
for the model completeness above we cannot drop $\preceq$  from the language $\mathcal{L}$.

\begin{cor}\label{DIVP-} Every newtonian $\upo$-free Liouville closed $H$-field has \rm{DIVP}.
\end{cor}
\begin{proof} Let $K$ be a newtonian $\upo$-free Liouville closed $H$-field.
If the derivation of $K$ is small, then DIVP follows from the results from~\cite{JvdH} quoted earlier and the above completeness result from~\cite{ADH}. 
Suppose the derivation of $K$ is not small. Replacing the derivation $\der$ of $K$ by
a multiple $\phi^{-1}\der$ with $\phi>0$ in $K$ transforms $K$ into its so-called compositional conjugate $K^{\phi}$, which is still a newtonian $\upo$-free Liouville closed $H$-field, and $K$ has DIVP iff $K^{\phi}$ does.  By  4.4.7 and 9.1.5 in~\cite{ADH} we can choose~$\phi>0$ in $K$ such that the derivation $\phi^{-1}\der$ of $K^\phi$ is small. 
\end{proof}


\noindent
This gives one direction of Theorem~\ref{DIVP}. In the rest of this paper 
we prove a strong version, Corollary~\ref{LIVP}, of the other direction, without using \cite{JvdH} but relying heavily on various parts of \cite{ADH} with detailed
references. 
Theorem~\ref{DIVP} and the results quoted above from~\cite{ADH} yield the result stated in the abstract: a Liouville closed $H$-field with small derivation is elementarily equivalent to $\T$ iff it has DIVP.

\subsection*{Connection to Hardy fields} Every Hardy field $H$ extends to a Hardy field $H(\R)\supseteq \R$, and $H(\R)$ is in particular an $H$-field. We refer to~\cite{ADH2} for
a discussion of the conjecture that {\em any Hardy field containing $\R$ extends to a newtonian $\upo$-free Hardy field}. At the end of 2019  we finished the proof of this conjecture by considerably refining
material in~\cite{ADH} and~\cite{vdH:hfsol}; this amounts to a rather complete extension theory of Hardy fields.  
Note that every Hardy field extends to a maximal Hardy field, by Zorn, and
so having established this conjecture we now know that all maximal Hardy fields are elementarily equivalent to $\T$, as ordered differential fields. Since $\C$ has the cardinality $\mathfrak c=2^{\aleph_0}$ of the continuum, there are at most $2^{\mathfrak c}$ many maximal Hardy fields, and we also have a proof that there are exactly that many. (We thank Ilijas Farah for a useful hint on this point.) These remarks on Hardy fields serve as an announcement. A rather voluminous work containing the proof of the conjecture is currently being prepared for publication. We also hope to include there a proof of DIVP for newtonian $\upo$-free $H$-fields that does not depend as in the present paper on it being true for $\T_{\operatorname{g}}$, whose proof in~\cite{JvdH} uses the particular nature of $\T_{\operatorname{g}}$.

We have a second conjecture about Hardy fields in \cite{ADH2}, whose proof is not yet finished at this time (May 2021): {\em for any maximal Hardy field $H$ and countable subsets~$A< B$ in $H$ there exists $y\in H$ such that $A<y<B$}.  This means that the underlying ordered set of a maximal Hardy field is an $\eta_1$-set in the sense of Hausdorff. Together with the (now established) first conjecture and results from~\cite{ADH} it implies: {\em all maximal Hardy fields are back-and-forth equivalent as ordered differential fields,
and thus isomorphic assuming {\rm{CH}}, the  Continuum Hypothesis}. 


\section{Preliminaries}

\noindent
In order to make free use of the valuation-theoretic tools from~\cite{ADH} and
to make this paper self-contained modulo references to specific results from the literature we provide more background in this section before returning to DIVP.

\subsection*{Notation and terminology}
Throughout, $m$,~$n$ range over $\N=\{0,1,2,\dots\}$.
Given an additively written abelian group $A$ we let $A^{\ne}:=A\setminus\{0\}$.
 Rings are commutative with  identity $1$, and for a ring~$R$ we let $R^\times$ be the multiplicative group of units (consisting of the $a\in R$ such that $ab=1$ for some $b\in R$). A {\em differential ring\/} will be a ring $R$ containing
(an isomorphic copy of) $\Q$ as a subring and equipped with a derivation~$\der\colon R \to R$; note that then $C_R:=\big\{a\in R:\, \der(a)=0\big\}$ is a subring of $R$, called the ring of constants of $R$, and that $\Q\subseteq C_R$.  If $R$ is a field, then so is $C_R$.
An ordered differential field is in particular a differential ring.

Let $R$ be a differential ring and $a\in R$. When its derivation $\der$ is clear from the context we denote $\der(a),\der^2(a),\dots,\der^n(a),\dots$ by $a', a'',\dots, a^{(n)},\dots$, and if $a\in R^\times$, then $a^\dagger$ denotes $a'/a$, so $(ab)^\dagger=a^\dagger + b^\dagger$ for $a,b\in R^\times$. In Section~\ref{hf} we need to consider the function 
$\omega=\omega_R\colon R \to R$ given by $\omega(z)=-2z'-z^2$, and
the function~$\sigma=\sigma_R\colon R^\times \to R$ given by
$\sigma(y)=\omega(z)+y^2$ for $z:= -y^\dagger$.

We have the differential ring $R\{Y\}=R[Y, Y', Y'',\dots]$ of differential polynomials in an indeterminate $Y$ over $R$.  
We say that  
$P=P(Y)\in R\{Y\}$ has order at most~$r\in \N$ if~$P\in R[Y,Y',\dots, Y^{(r)}]$. 

For $\phi\in R^\times$ we let $R^{\phi}$ be the {\it compositional conjugate of~$R$ by~$\phi$}\/: the differential ring
with the same underlying ring as~$R$ but with derivation~$\phi^{-1}\der$ instead of $\der$. 
We then have an $R$-algebra isomorphism
$$P\mapsto P^\phi\ \colon\ R\{Y\}\to R^\phi\{Y\}$$ with $P^\phi(y)=P(y)$ for all $y\in R$;
see \cite[Sec\-tion~5.7]{ADH}.

For a field $K$ we have $K^\times=K^{\ne}$, and
a (Krull) valuation on $K$ is a surjective map 
$v\colon K^\times \to \Gamma$ onto an ordered abelian 
group $\Gamma$ (additively written) satisfying the usual laws, and extended to
$v\colon K \to \Gamma_{\infty}:=\Gamma\cup\{\infty\}$ by $v(0):=\infty$,
where the ordering on $\Gamma$ is extended to a total ordering
on $\Gamma_{\infty}$ by $\gamma<\infty$ for all 
$\gamma\in \Gamma$. 

Let $K$ be a {\em valued field\/}: a field (also denoted by $K$) together with a valuation ring~$\mathcal O$ of that field. This yields a valuation $v\colon K^\times \to \Gamma$  on the underlying field
such that $\mathcal O=\{a\in K:va\geq 0\}$ as explained in~\cite[Section~3.1]{ADH}.
We introduce various binary relations 
on the set $K$ by defining for $a,b\in K$:  
\begin{align*} a\asymp b &\ :\Leftrightarrow\ va =vb, & a\preceq b&\ :\Leftrightarrow\ va\ge vb, & a\prec b &\ :\Leftrightarrow\  va>vb,\\
a\succeq b &\ :\Leftrightarrow\ b \preceq a, &
a\succ b &\ :\Leftrightarrow\ b\prec a, & a\sim b &\ :\Leftrightarrow\ a-b\prec a.
\end{align*}
It is easy to check that if $a\sim b$, then $a, b\ne 0$, and that
$\sim$ is an equivalence relation on $K^\times$. We also let $\smallo=\{a\in K:va>0\}$ be the maximal ideal of $\mathcal{O}$, so~$\mathcal{O}/\smallo$ is the residue field
of the valued field $K$.
A convex subgroup $\Delta$ of the value group $\Gamma$ of $v$ gives rise to the
{\it $\Delta$-coarsening}\/ of the valued field $K$; see [ADH, 3.4].

\subsection*{$H$-fields and pre-$H$-fields} As in \cite{ADH},  a {\em valued differential field\/} is a valued field $K$ with residue field of characteristic zero and equipped with a derivation 
$\der\colon K \to K$. An 
 {\em ordered valued differential field\/} is a valued differential field $K$ equipped with an ordering on $K$ making $K$ an ordered field. We consider any $H$-field $K$ as an ordered valued differential field whose valuation ring is the convex hull in $K$ of its constant field $C$, in accordance with construing it as an $\mathcal{L}$-structure as specified in the introduction.
 
  A {\em pre-$H$-field\/} is by definition an ordered valued differential subfield of an $H$-field. By~\cite[Sections 10.1, 10.3, 10.5]{ADH}, an ordered valued differential field $K$
  is a pre-$H$-field iff the valuation ring $\mathcal{O}$ of $K$ is convex in $K$, $f'>0$ for all $f>\mathcal{O}$ in~$K$, and $f'\prec g^\dagger$ for all $f,g\in K^\times$ with $f\preceq 1$ and $g\prec 1$. Any Hardy field $H$ is construed as a pre-$H$-field by taking the convex hull of $\Q$ in $H$ as its valuation ring, giving rise to the so-called ``natural valuation''
  on $H$ as an ordered field. At the end of Section~9.1 in \cite{ADH} we give 
  $\Q(\sqrt{2+x^{-1}})$ as an example of a Hardy field that is not an $H$-field.
 Any ordered differential field $K$ with the trivial valuation ring $\mathcal{O}=K$ is a pre-$H$-field (so the valuation ring of a pre-$H$-field $K$ is not always the convex hull
 in $K$ of its constant field, in contrast to Hardy fields and $H$-fields). If $K$ is a pre-$H$-field whose valuation ring is nontrivial, then the valuation topology on $K$ equals its order topology, by \cite[Lemma 2.4.1]{ADH}. 

 Let $K$ be a pre-$H$-field. Then the derivation of $K$ and its valuation $v\colon K^\times \to \Gamma$ induce an operation $\psi \colon  \Gamma^{\ne} \to \Gamma$, given by
 $\psi(vf)=v(f^\dagger)$ for $f\nasymp 1$ in $K^\times$; the pair~$(\Gamma,\psi)$ is called the $H$-asymptotic couple of $K$; see \cite[Section~9.1]{ADH}.  Below we assume
some familiarity with $(\Gamma,\psi)$, and
properties of $K$ based on it, such as~$K$ having {\it asymptotic integration}\/ and~$K$ having a {\it gap}\/ \cite[Sections~9.1, 9.2]{ADH}.
The {\it flattening}\/ of $K$ is the $\Gamma^\flat$-coarsening of $K$ where $\Gamma^\flat=\{vf :\, f\in K^\times,\ f' \prec f\}$,
with associated binary relations $\asymp^\flat$, $\preceq^\flat$ etc.; see [ADH, 9.4].

\section{DIVP}\label{hf}

\noindent
In this section $K$ is a pre-$H$-field. We let $\mathcal{O}$ be its valuation ring, with maximal ideal $\smallo$, and corresponding valuation $v\colon K^\times \to \Gamma=v(K^\times)$. Let $(\Gamma,\psi)$ be its $H$-asymptotic couple, and $\Psi:=\big\{\psi(\gamma):\gamma\in\Gamma^{\ne}\big\}$. Recall that ``$K$ has DIVP'' means: for all
$P(Y)\in K\{Y\}$ and $f< g$ in $K$ with $P(f) <0 <  P(g)$ there is a~$y\in K$ such that $f < y < g$ and~$P(y)=0$. Restricting this to $P$ of order~$\le r$, where~$r\in \N$, gives the notion of~$r$-DIVP. Thus
$K$ having $0$-DIVP is equivalent to~$K$ being real closed as an
ordered field. In particular, if $K$ has $0$-DIVP, then $\Gamma=v(K^\times)$ is divisible. 
From~\cite[Section~2.4]{ADH} recall our convention that $K^>=\{a\in K:a>0\}$, and similarly with~$<$ replacing $>$.

\begin{lemma}\label{ivp1} Suppose $\Gamma\ne \{0\}$ and $K$ has 
$1$-$\operatorname{DIVP}$. Then
$\der K=K$, $(K^{>})^\dagger=(K^{<})^\dagger$ is a convex subgroup of $K$, $\Psi$ has no largest element, and $\Psi$ is convex in $\Gamma$.
\end{lemma}
\begin{proof} We have $y'=0$ for $y=0$, and $y'$ takes arbitrarily large positive values in~$K$ as $y$ ranges over
$K^{>\mathcal{O}}=\{a\in K:a>\mathcal O\}$, since by \cite[Lemma~9.2.6]{ADH} the set
$(\Gamma^{<})'$ is coinitial in~$\Gamma$. Hence
$y'$ takes all positive values on $K^{>}$, and therefore also all negative values on~$K^{<}$. Thus $\der K=K$. Next, let $a,b\in K^{>}$, and suppose
$s\in K$ lies strictly between
$a^\dagger$ and $b^\dagger$. Then $s=y^\dagger$ for some $y\in K^{>}$
strictly between $a$ and $b$; this follows by noting that for
$y=a$ and $y=b$ the signs of
$sy-y'$ are opposite.

Let $\beta\in\Psi$ 
and take $a\in K$ with $v(a')=\beta$. Then $a\succ 1$, since
$a\preceq 1$ would give~$v(a')>\Psi$. Hence for $\alpha=va<0$
we have $\alpha+\alpha^\dagger=\beta$, so $\alpha^\dagger>\beta$.
Thus $\Psi$ has no largest element. 
Therefore the set $\Psi$ is convex in $\Gamma$.  \end{proof}  

\noindent
Thus the ordered differential field $\T_{\log}$ of   logarithmic transseries~\cite[Appendix~A]{ADH} does not have $1$-DIVP, although it is a newtonian 
$\upo$-free $H$-field.

Does DIVP imply that $K$ is an $H$-field? No: take an $\aleph_0$-saturated elementary
extension of $\T$ and let $\Delta$ be as in \cite[Example~10.1.7]{ADH}. Then
the $\Delta$-coarsening of~$K$ is a pre-$H$-field with DIVP and nontrivial value group, and has a gap, but it is not an $H$-field.
On the other hand:

\begin{lemma}\label{ivp2} Suppose $K$ has $1$-$\operatorname{DIVP}$ and has no gap. Then $K$ is an $H$-field.
\end{lemma}
\begin{proof} In \cite[Section~11.8]{ADH} we defined
$$\text{I}(K)\ :=\  \{y\in K:\, \text{$y\preceq f'$ for some $f\in \mathcal{O}$}\},$$
a convex $\mathcal{O}$-submodule of $K$. Since $K$ has no gap, we have
$$\der\smallo\ \subseteq\ \text{I}(K)\ =\ \{y\in K:\, \text{$y\preceq f'$ for some $f\in \smallo$}\}.$$
Also $\Gamma\ne \{0\}$, and so $(\Gamma,\psi)$ has asymptotic integration by Lemma~\ref{ivp1}. We show that $K$ is an $H$-field
by proving $\text{I}(K) =\der\smallo$, so let $g\in \text{I}(K)$, $g<0$. Since $(\Gamma^{>})'$ has no least element we can take
positive $f\in \smallo$ such that $f'\succ g$. Since
$f'<0$, this gives $f' < g$. Since $(\Gamma^{>})'$ is cofinal in 
$\Gamma$ we can also take positive $h\in \smallo$ such that~$h'\prec g$, which in view of $h'<0$ gives $g < h'$.
Thus $f' < g < h'$, and so $1$-DIVP yields $a\in \smallo$ with $g=a'$. 
\end{proof}

\noindent
We refer to Sections~11.6 and~14.2 of \cite{ADH} for the definitions of {\it $\upl$-freeness}\/ and {\it $r$-newtonianity}\/ ($r\in\N$). From the introduction we recall that
$\omega(z):=-2z'-z^2$. Below, compositionally conjugating an $H$-field $K$ means replacing it by some $K^\phi$ with $\phi\in K^{>}$; this preserves most relevant properties like being an $H$-field, being $\upl$-free, $r$-DIVP, and 
$r$-newtonianity, and it replaces $\Psi$ by $\Psi-v\phi$. 

\begin{lemma}\label{ivp3} Suppose $K$ is an $H$-field, $\Gamma\ne \{0\}$, and $K$ has 
$1$-$\operatorname{DIVP}$. Then $K$ is $\upl$-free and $1$-newtonian, and the subset $\omega(K)$ of $K$ is downward closed.  
\end{lemma}
\begin{proof} Note that $K$ has (asymptotic) integration, by Lemma~\ref{ivp1}. 
Assume towards a contradiction that $K$ is
not $\upl$-free. We arrange by compositional conjugation that~$K$ has small derivation, so $K$ has an element $x\succ 1$ with $x'=1$, hence $x> C$. A construction in the beginning of \cite[Section 11.5]{ADH} yields by \cite[Lemma 11.5.2]{ADH} a pseudocauchy sequence 
$(\upl_{\rho})$ in $K$ with certain properties including $\upl_{\rho}\sim x^{-1}$ for all $\rho$. 
As $K$ is not $\upl$-free, $(\upl_{\rho})$ has a pseudolimit $\upl\in K$ by \cite[Corollary~11.6.1]{ADH}. Then $s:= -\upl\sim -x^{-1}$, and
$s$ creates a gap
over~$K$ by \cite[Lemma~11.5.14]{ADH}. Now note that for
$P:= Y'+sY$ we have $P(0)=0$ and $P(x^2)=2x+sx^2\sim x$, so
by $1$-DIVP we have $P(y)=1$ for some $y\in K$, contradicting
\cite[Lemma~11.5.12]{ADH}. 

Let $P\in K\{Y\}$ of order at most~$1$ have Newton degree~$1$; we have to show that $P$ has a zero in $\mathcal{O}$. We know that $K$ is $\upl$-free, so by \cite[Proposition~13.3.6]{ADH} we
can pass to an elementary extension, 
compositionally conjugate, and divide by an element of $K^\times$ to arrange that $K$ has small derivation and
$P=D+R$ where $D=cY+d$ or~$D=cY'$ with $c,d\in C$, $c\ne 0$, and where $R\prec^{\flat} 1$. Then $R(a)\prec^{\flat} 1$ for all $a\in \mathcal{O}$. If $D= cY+d$, then we can take $a,b\in C$ with $D(a)<0$ and $D(b)>0$, which in view of $R(a)\prec D(a)$ and
$R(b) \prec D(b)$ gives $P(a)<0$ and $P(b)>0$, and so
$P$ has a zero strictly between $a$ and $b$, and thus a zero in $\mathcal{O}$. Next, suppose $D=cY'$. Then we take $t\in \smallo^{\ne}$ with $v(t^\dagger)=v(t)$, that is, $t'\asymp t^2$, so
$$P(t)\ =\ ct'+R(t), \quad P(-t)\ =\ -ct'+ R(-t), \qquad 
R(t),\ R(-t)\ \prec\ t'.$$
Hence $P(t)$ and $P(-t)$ have opposite signs, so
$P$ has a zero strictly between $t$ and~$-t$, and thus $P$
has a zero in $\mathcal{O}$.

From $\omega(z)=-z^2-2z'$ we see that $\omega(z) \to -\infty$ as
$z\to +\infty$ and as $z\to -\infty$ in~$K$, so $\omega(K)$ is downward closed by $1$-IVP.  
\end{proof} 

\noindent
For results involving $r$-DIVP for $r\ge 2$ we need a variant of \cite[Lemma~11.8.31]{ADH}. To state this variant we introduce as in~\cite[Section~11.8]{ADH} the sets
$$\Upg(K)\ :=\ \{ a^\dagger:\, a\in K\setminus\mathcal O\}\ \subseteq\ K^>, \qquad \Upl(K)\ :=\ -\Upg(K)^\dagger\ \subseteq\ K.$$
The superscripts~$\uparrow$,~$\downarrow$ used in the statement of Lemma~\ref{ivp4-} below indicate upward, respectively downward, closure
in the ordered set $K$, as in~\cite[Section~2.1]{ADH}.

\begin{lemma}\label{ivp4-} Let $K$ be an $H$-field with asymptotic integration. Then 
$$K^{>}\ =\ \operatorname{I}(K)^{>}\cup\Upg(K)^{\uparrow}, \qquad \sigma\big(K^{>}\setminus \Upg(K)^{\uparrow}\big)\ \subseteq\ \omega\big(\Upl(K)\big){}^{\downarrow}.$$ 
\end{lemma}
\begin{proof} If $a\in K$, $a > \text{I}(K)$, then $a\ge b^\dagger$ for some $b\in K^{\succ 1}$, and thus $a\in \Upg(K)^{\uparrow}$. Next, let $s\in K^{>}\setminus\Upg(K)^{\uparrow}$;
we have to show $\sigma(s)\in \omega\big(\Upl(K)\big){}^{\downarrow}$. Note that $s\in \operatorname{I}(K)^{>}$ by what we just proved. From~\cite[10.2.7 and  10.5.8]{ADH} we obtain an immediate $H$-field extension $L$ of $K$ and $a\in L^{\succ 1}$ with $s=(1/a)'$. As in the proof 
of~\cite[11.8.31]{ADH} with $L$ instead of $K$ this gives $\sigma(s)\in \omega\big(\Upl(L)\big){}^{\downarrow}$, where $\downarrow$ indicates here the downward closure in $L$. 
It remains to note that $\omega$ is increasing on $\Upl(L)$ by the remark preceding~\cite[11.8.21]{ADH},
and that $\Upl(K)$ is cofinal in $\Upl(L)$ by~\cite[11.8.14]{ADH}. 
\end{proof} 

\noindent
The concept of \emph{$\upo$-freeness}\/ is introduced in \cite[Section~11.7]{ADH}.  
If $K$ has asymptotic integration, then by~\cite[11.8.30]{ADH}: $K$ is $\upo$-free\  
$\Leftrightarrow\ K=\omega\big(\Upl(K)\big){}^\downarrow\cup \sigma\big(\Upg(K)\big){}^\uparrow$. 

The next lemma also mentions the differential field extension $K[\imag]$ of $K$
where $\imag^2=-1$, as well as linear differential operators over  differential fields like $K$ and~$K[\imag]$; for this we refer to \cite[Sections 5.1, 5.2]{ADH}. 

\begin{lemma}\label{ivp4} Suppose $K$ is an $H$-field, $\Gamma\ne \{0\}$, $r\ge 2$, and $K$ has $r$-$\operatorname{DIVP}$. Then the following hold, with {\rm{(i), (ii), (iii)}} using only the case $r=2$: 
\begin{enumerate}
\item[(i)]
$K=\omega(K)\cup \sigma(K^{>})=\omega\big(\Upl(K)\big){}^\downarrow\cup \sigma\big(\Upg(K)\big){}^\uparrow$;
\item[(ii)] $K$ is $\upo$-free and $\omega(K)=\omega\big(\Upl(K)\big){}^\downarrow$;
\item[(iii)]  for all $a\in K$ the operator $\der^2-a$ splits over $K[\imag]$;
\item[(iv)] $K$ is $r$-newtonian.
\end{enumerate}
\end{lemma}
\begin{proof} For (i) we use the end of~\cite[Section~11.7]{ADH} to replace~$K$ with a compositional conjugate so that $0\in \Psi$. Then $K$ has small derivation,
 and we have $a\in K^{>}$ such that $a\nasymp 1$ and $a^\dagger\asymp 1$. Replacing $a$ by $a^{-1}$ if necessary this gives
  $a^\dagger=-\phi$ with~$\phi\asymp 1$, $\phi>0$, so $a\prec 1$. Then $\phi^{-1}a^\dagger=-1$; replacing $K$ by $K^\phi$ and renaming the latter as $K$ this means  $a^\dagger=-1$. Let $f\in K$; to get~$f\in \omega\big(\Upl(K)\big){}^\downarrow\cup \sigma\big(\Upg(K)\big){}^\uparrow$, note first that $1= (1/a)^\dagger\in \Upg(K)$, so 
  $0\in \Upl(K)$. Also $\omega\big(\Upl(K)\big){}^\downarrow\subseteq \omega(K)$ by Lemma~\ref{ivp3}.

If $f\le 0$, then $\omega(0)=0$ gives $f\in \omega\big(\Upl(K)\big){}^\downarrow$.
So assume~$f> 0$;
we first show that then~$f\in \sigma(K^{>})$. Now  for $y\in K^{>}$, $f=\sigma(y)$ is equivalent (by multiplying with $y^2$) to $P(y)=0$, where $$P(Y)\ :=\ 2YY'' -3(Y')^2+Y^4-fY^2\in K\{Y\}.$$
See also \cite[Section~13.7]{ADH}. We have $P(0)=0$ and $P(y) \to +\infty$ as $y\to +\infty$ (because of the term $y^4$). In view of $2$-DIVP it will suffice to show that
for some $y>0$ in $K$ we have $P(y)<0$. Now with
$y\in K^{>}$ and $z:=-y^\dagger$ we have
\begin{align*} P(y)\ &=\ y^2\big(\sigma(y)-f\big)\ =\ y^2\big(\omega(z)+y^2-f\big),\text{ hence}\\ 
P(a)\ &=\ a^2\big(\omega(1)+a^2-f\big)\ =\ a^2(-1+a^2-f)\ <\ 0,
\end{align*}
so $f\in \sigma(K^{>})$. By the second inclusion of Lemma~\ref{ivp4-}
this yields $f\in \omega\big(\Upl(K)\big){}^\downarrow$ or 
$f\in \sigma\big(\Upg(K)^\uparrow\big)$. But we have $\sigma\big(\Upg(K)^\uparrow\big)\subseteq \sigma\big(\Upg(K)\big){}^\uparrow$, because $\sigma$ is increasing on~$\Upg(K)^\uparrow$ by the remark preceding~\cite[11.8.30]{ADH}. This concludes the proof of (i), and then
(ii) follows, using for its second part also the fact  stated just before 
\cite[11.8.29]{ADH} that we have $\omega(K) < \sigma\big(\Upg(K)\big)$. 

Now (iii) follows from $K=\omega(K)\cup \sigma(K^{>})$ by \cite[Section 5.2, (5.2.1)]{ADH}. As to (iv), 
let $P\in K\{Y\}$ of order at most $r$ have Newton degree $1$; we have to show that $P$ has a zero in $\mathcal{O}$.  
For this we repeat the argument in the proof of Lemma~\ref{ivp3} 
so that it applies to our~$P$, using $\upo$-freeness instead of $\upl$-freeness, \cite[Proposition~13.3.13]{ADH} instead of \cite[Proposition~13.3.6]{ADH}, and $r$-DIVP instead of $1$-DIVP. 
\end{proof}

\begin{cor}\label{LIVP} If $K$ is an $H$-field, $\Gamma\ne \{0\}$, and $K$ has 
$\operatorname{DIVP}$, then $K$ is $\upo$-free and newtonian. 
\end{cor} 

\noindent
There are non-Liouville closed $H$-fields with nontrivial derivation that have~$\operatorname{DIVP}$; see \cite[Section~14]{AvdD2}. By Lemma~\ref{ivp3} and Lemma~\ref{ivp4}(iii), Liouville closed $H$-fields having $2$-$\operatorname{DIVP}$ are {\it Schwarz closed}\/
as defined in~\cite[Sec\-tion~11.8]{ADH}.

\begin{theorem}\label{DIVP} Let $K$ be a Liouville closed $H$-field. Then 
$$ \text{$K$   has $\operatorname{DIVP}$}\ \Longleftrightarrow\ \text{$K$  is $\upo$-free and newtonian.}$$
\end{theorem}
\begin{proof} 
The forward direction is part of Corollary~\ref{LIVP}. The backward direction is Corollary~\ref{DIVP-}. 
\end{proof}



\end{document}